\documentclass[12pt,a4paper]{amsart}
\usepackage[english]{babel}
\usepackage[latin1]{inputenc}
\usepackage{graphicx,tikz}
\usepackage{amssymb, amsmath}
\usepackage{geometry,enumerate}
\usepackage{enumitem,centernot,color}
\usepackage[colorlinks=true,linkcolor=blue,urlcolor=blue,citecolor=blue]{hyperref}
\newtheorem{theorem}{Theorem}[section]
\newtheorem{definition}[theorem]{Definition}
\newtheorem{lemma}[theorem]{Lemma}

\newtheorem{proposition}[theorem]{Proposition}

\theoremstyle{definition}

\newtheorem{examples}[theorem]{Examples}
\newtheorem{remark}[theorem]{Remark}
\newtheorem{remarks}[theorem]{Remarks}
\newtheorem{problem}[theorem]{Problem}

\newcommand{\N}{\mathbb{N}}

\newcommand{\R}{\mathbb{R}}

\def\N{\mathbb N}

\def\R{\mathbb R}

\def\syd{\,\Delta\,}
\def\xnj{X_j^n}
\def\xnjn{X^n_{j_n}}
\def\xnk{X_k^n}
\def\yni{Y_i^n}
\def\qij{Q^n_{ij}}
\def\qik{Q^n_{ik}}

\def\eps{\varepsilon}

\def\f{f^*}

\newcommand{\ds}{\mathrm{\,d}s}

\newcommand{\dx}{\mathrm{\,d}x}
\newcommand{\dy}{\mathrm{\,d}y}
\newcommand{\dz}{\mathrm{\,d}z}
\newcommand{\dmu}{\mathrm{\,d}\mu}
\newcommand{\dnu}{\mathrm{\,d}\nu}

\begin{document}

\title[Lorentz and Gale-Ryser theorems on general measure spaces]{Lorentz and Gale-Ryser theorems  \\ on general measure spaces}

\author[S. Boza]{Santiago Boza$^*$}

\address{Department of  Mathematics, EETAC, Polytechnical University of Catalonia, 08860 Castelldefels, Spain.}
\email{santiago.boza@upc.edu}

\author[M.  K\v repela]{Martin K\v repela$^{**}$}
\address{Czech Technical University in Prague, Faculty of Electrical Engineering, Department of Mathematics, Technick\'a~2, 166~27 Praha~6, Czech Republic.}
\email{martin.krepela@fel.cvut.cz}

\author[J. Soria]{Javier Soria$^{*}$}

\address{Interdisciplinary Mathematics Institute (IMI), Department of Analysis and Applied Mathematics, Complutense University  of Madrid, 28040 Madrid, Spain.}
\email{javier.soria@ucm.es}

\thanks{$^*$The author was partially  supported by the Spanish Government grant MTM2016-75196-P (MINECO / FEDER, UE).\\
$^{**}$The author was supported by the project OPVVV CAAS CZ.02.1.01/0.0/0.0/16\_019/0000778.
}



\begin{abstract}
Based on the Gale-Ryser theorem \cite{ga:ga,ry:ry}, for the existence of suitable $(0,1)$-matrices for different partitions of a natural number, we revisit the classical result of G. G. Lorentz \cite{lor} regarding the characterization of a plane measurable set, in terms of its cross sections, and extend it to general measure spaces.
\end{abstract}

\maketitle


\section{Introduction}
In \cite{lor}, G. G. Lorentz fully characterized the existence of a plane set in terms of its cross section. The main result reads as follows:

\begin{theorem}\label{planeset}
Suppose that $P(x)$, $Q(y)$ are two non-negative integrable functions defined for $- \infty <x < + \infty$, $- \infty < y < + \infty$. In order that
a measurable set $A$ with cross functions $P(x)$, $Q(y)$ exists, it is necessary
and sufficient that the non-increasing rearrangements $p(x)$, $q (y)$ of these functions satisfy the conditions:
\begin{align}
\int_0^xp(u)\,du&\le\int_0^xq^{-1}(u)\,du,\qquad x>0,\label{pq}\\
\int_0^xq(u)\,du&\le\int_0^xp^{-1}(u)\,du,\qquad x>0.\label{qp}
\end{align}
\end{theorem}

In modern terminology, the nonincreasing rearrangement of a function $f$ on a measurable space $(X,\mu)$ is defined as 
$$
f^*(t)=\inf\{s>0:\lambda_f(s)\le t\},
$$
where
$$
\lambda_f(s)=\mu(\{x\in X:|f(y)|>s\})
$$
is the distribution function of $f$  (see \cite{BS} for standard definitions and classical properties in this setting). It is worth to mention that, according to Lorentz's notation, we have that $p^{-1}(u)=\lambda_p(u)$. It is also proved in \cite{lor} that \eqref{pq} and \eqref{qp} are equivalent to \eqref{pq} and the condition $\|P\|_1=\|Q\|_1$.

\medskip

A few years later, D. Gale and H. J. Ryser, studying some graph theoretical  conditions for degree sequences on simple graphs \cite{sh:sh}, proved in  \cite{ga:ga,ry:ry} a discrete version, namely, they characterized the existence of a $(0,1)$-matrix $A$, with predetermined $r(A)$, the sums of its rows,  and $c(A)$, the sums of its columns  (which corresponds to fixing  2  partitions of   a given $n\in\mathbb N$). For example, if $n=5=3+2=2+2+1$, then the matrix
\begin{equation}\label{matrix}
A=\left(\begin{array}{ccc}1 & 1 & 1 \\1 & 1 & 0\end{array}\right)
\end{equation}
satisfies that   $r(A)=\{3,2\}$ and for the columns we obtain $c(A)=\{2,2,1\}$, as desired. Similarly, it is easy to see that for the partitions $n=5=4+1=2+2+1$ there is not such a matrix. The aforementioned characterization is given as follows:

\begin{theorem}[Gale--Ryser, \cite{ga:ga,ry:ry}]\label{gary}
Let $p=\{p_1,\dots,p_j\}$ and $q=\{q_1,\dots,q_k\}$ be two nonincreasing partitions of a positive integer (i.e., $p,q\subset\mathbb N$, $p_1\ge\cdots\ge p_j$,  $q_1\ge\cdots\ge q_k$ and $p_1+\cdots+p_j=q_1+\cdots+q_k$). Then, there exists a $(0,1)$-matrix $A\in\mathcal M_{j\times k}$ such that $r(A)=p$ and $c(A)=q$ if and only if for all positive integers $m$,
$$
\sum_{i=1}^mq_i\le \sum_{i=1}^m\widehat{p_i},
$$
where $\widehat{p_i}=\operatorname{card}\{1\le l\le j:p_l\ge i\}$, if $1\le i\le p_1$, $\widehat{p_i}=0$, if $i>p_1$ and $q_i=0$, if $i>k$.
\end{theorem}

As we can readily see, these conditions are in exact analogy with those in Theorem~\ref{planeset}. Influenced by this matricial case, we find a new approach (using the geometrical definition of horizontal swappable squares), allowing us to extend  and unify both Theorems~\ref{planeset} and \ref{gary} by considering products of general resonant measure spaces (see Theorem~\ref{T:GR}):

\begin{problem}\label{conj}
Let $(X,\mu)$ and $(Y,\nu)$ be two $\sigma$-finite measure spaces and let $f:X\rightarrow \mathbb R^+$ and $g:Y\rightarrow \mathbb R^+$ be measurable functions. Does there exist a measurable subset $E\subset X\times Y$ such that $f$ and $g$ are the corresponding cross sections of $E$:
\begin{equation}\label{setfg}
f(x)=\int_{Y}\chi_E(x,y)\,{\rm d}\nu(y)\quad\text{and}\quad g(y)=\int_{X}\chi_E(x,y)\,{\rm d}\mu(x),
\end{equation}
$\mu$-a.e. $x\in X$ and $\nu$-a.e. $y\in Y$?
\end{problem}

It is clear that a necessary condition for \eqref{setfg} to hold is that $\|f\|_{L^1(\dmu)}=\|g\|_{L^1(\dnu)}$. However, as the previous matricial example shows for   counting finite measures, this equality is not enough.

The following result gives the key  estimate to solve Problem~\ref{conj} (which will be considered in Section~\ref{charact}).  For the rest of this section we will assume the following \textit{resonant} condition on  the  measure space $X$, which is equivalent to saying that $X$ is either   nonatomic or completely atomic, with all atoms having equal measure \cite[Theorem~II.2.7]{BS}.

\begin{proposition}\label{hlpineq}
Let $(X,\mu)$ and $(Y,\nu)$ be two $\sigma$-finite  measure spaces, assume $X$ to be   resonant and  let $E\subset X\times Y$ be a measurable subset. Let $f$ and $g$ be the cross sections of $E$ on $X$ and $Y$, respectively, as in \eqref{setfg}. Then, for every $t>0$  we have that
\begin{equation}\label{HLP}
\int_0^tf^*(s)\,\ds\le \int_0^t\lambda_g(s)\,\ds.
\end{equation}
\end{proposition}

\begin{proof}
We first assume that $t>0$ is in the range of $\mu$, and choose any set $A_t\subset X$, with $\mu(A_t)=t$. Then, setting
$$
E_x=\{y\in Y:(x,y)\in E\}\quad\text{and}\quad E^y=\{x\in X: (x,y)\in E\},
$$
and using \cite[Lemma~II.2.1]{BS}, we have that

\begin{align*}
\int_{A_t}f(x)\,{\rm d}\mu(x)&=\int_{A_t}\int_{Y}\chi_E(x,y)\,{\rm d}\nu(y)\,{\rm d}\mu(x)=\int_{Y}\int_{A_t}\chi_E(x,y)\,{\rm d}\mu(x)\,{\rm d}\nu(y)\\
&=\int_{Y}\int_{A_t}\chi_{E^y}(x)\,{\rm d}\mu(x)\,{\rm d}\nu(y)\le\int_{Y}\int_{0}^{t}\chi_{(0,\mu(E^y))}(s)\,{\rm d}s\,{\rm d}\nu(y)\\
&=\int_{Y}\int_{0}^{t}\chi_{(0,g(y))}(s)\,{\rm d}s\,{\rm d}\nu(y)=\int_{0}^{t}\int_{Y}\chi_{(0,g(y))}(s)\,{\rm d}\nu(y)\,\ds\\
&=\int_{0}^{t}\int_{Y}\chi_{\{y\in Y:g(y)>s\}}\,{\rm d}\nu(y)\,ds=\int_0^t\lambda_g(s)\,\ds.
\end{align*}

Finally, since $X$ is resonant and using  \cite[Proposition~II.3.3]{BS}

$$
\int_0^tf^*(s)\,\ds=\sup_{\{A_t\subset X:\mu(A_t)=t\}}\int_{A_t}f(x)\,{\rm d}\mu(x)\le \int_0^t\lambda_g(s)\,\ds.
$$
Once we have proved \eqref{HLP} for $t>0$ in the range of $\mu$, we observe that the inequality trivially holds for any $t\ge\mu(X)$, since then both sides of \eqref{HLP} are equal to  the measure of $E$. Now, if $0<t<\mu(X)$  and $X$ is nonatomic, then we can find a measurable subset $A\subset X$ such that $\mu(A)=t$ (see \cite{sp:sp}) and we are done. To finish,  if $0<t<\mu(X)$ and $X$ is a discrete (totally atomic) measure space, then $\int_0^tf^*(s)\,ds$  is a piecewise linear concave function and  $\int_0^t\lambda_g(s)\,ds$ is a  concave function  greater than the previous integral at the nodes (the $\mu$-measure of  a finite collection of atoms). Hence, by the concavity property, the inequality is also true for the intermediate values of $t>0$.
\end{proof}

\begin{remarks}\label{remarks}

A simple remark, when we work with arbitrary general measures, is that \eqref{setfg} implies that the cross sections must take values in the image of the measure. For example, if $X=\{1,2\}$ and $Y=\{1\}$, both with the cardinality measures, then $f:X\rightarrow\R^+$, $f(1)=f(2)=1/2$ and $g:Y\rightarrow\R^+$, $g(1)=1$ satisfy \eqref{HLP}, but $f$ is not the cross section of any set $E\subset X\times Y$, since the cardinality measure only takes nonnegative integer values.

\medskip

We observe that \eqref{HLP} is not a homogeneous inequality. For example, if  $f\equiv g\equiv 1$ on $[0,1]$, corresponding to the case $E=[0,1]\times[0,1]$, then \eqref{HLP} trivially holds, but it is false for $2f$ and $2g$.
\medskip
\end{remarks}

Condition  \eqref{HLP} is not, a priori,  symmetric on $f$ and $g$. However, we are going to prove, in Proposition~\ref{SstarSinfty},  that we can reverse the role of $f$ and $g$ in \eqref{HLP}, as long as they have the same $L^1$-norms, which, as we already know, is a necessary condition to solve Problem~\ref{conj} (to simplify the proof, continuity of the nonincreasing rearrangements will be also assumed).  We start by recalling some well-known equalities:

\begin{lemma}
\label{threeinone}
Let $(X, \mu)$ be a $\sigma$-finite measure spaces and let $f:X\rightarrow \R^+$ be  a measurable function, with $\|f\|_{L^1(\dmu)}<+\infty$. Then, for every $t>0$,
\begin{equation}
\label{f1}
  \int_t^{\infty} \lambda_f(s) \, \ds+t\lambda_f(t)=\int_0^{\lambda_f(t)} f^*(s) \, \ds
 \end{equation}
and
\begin{equation}
\label{f2}
\int_t^{\infty} f^*(s)\,\ds+t f^*(t)=\int_0^{f^*(t)} \lambda_f(s) \, \ds. 
\end{equation}
\end{lemma}

\begin{proposition}
\label{SstarSinfty}
Let $(X, \mu)$ and $(Y,\nu)$ be two $\sigma$-finite measure spaces and suppose that $f:X\rightarrow \R^+$ and $g:Y\rightarrow \R^+$ are measurable functions such that $f^*$ and $g^*$ are continuous   and   $\|f\|_{L^1({\rm d}\mu)}=\|g\|_{L^1({\rm d}\nu)}<\infty$. If \eqref{HLP} holds, for  $t>0$, then for every $r>0$, 
\begin{equation}\label{goal}
\int_0^r g^*(s) \, \ds \leq \int_0^r \lambda_f(s) \, \ds.
\end{equation}
\end{proposition}

\begin{proof}
Let us assume that \eqref{HLP} holds, for   $t>0$.
Since
$$
\|f\|_{L^1({\rm d}\mu)}=\int_0^\infty f^*(s)\,\ds=\int_0^\infty\lambda_g(s)\,\ds=\|g\|_{L^1({\rm d}\nu)},
$$
it is easy to see that \eqref{HLP} is equivalent to the inequality
\begin{equation}
\label{atinfty}
\int_t^{\infty} \lambda_g(s) \, \ds \leq \int_t^{\infty} f^*(s) \, \ds.
\end{equation}
Let us now prove \eqref{goal}. Using the hypothesis on $f$ and $g$, and Lemma~\ref{threeinone}, condition (\ref{atinfty}) is equivalent to the following inequality:
\begin{equation}
\label{equiv}
\int_0^{\lambda_g(t)} g^*(s) \, \ds -t\lambda_g(t) \leq \int_0^{f^*(t)} \lambda_f(s) \, \ds-tf^*(t).
\end{equation}
To prove (\ref{goal}), we observe that if $r\ge\|f\|_\infty$, the result is trivial since the right-hand side is equal to $\|f\|_1$. Now, if $0<r<\|f\|_\infty=f^*(0)$, then by the continuity of $f^*$ and the fact that $\lim_{s\to\infty}f^*(s)=0$, there exists a $t>0$ such that $r=f^*(t)$.
Let us distinguish the following two possibilities:
\medskip

If $f^*(t)\leq \lambda_g(t)$, then $\lambda_g(t)>0$ and hence $0<t<\|g\|_\infty$. Now,  since $g^*$ is nonincreasing, then  (\ref{equiv}) implies
$$
\begin{aligned}
\int_0^{f^*(t)} g^*(s) \, \ds-\int_0^{f^*(t)} \lambda_f(s) \, \ds& \leq t(\lambda_g(t)-f^*(t))-\int_{f^*(t)}^{\lambda_g(t)} g^*(s) \, \ds\\
&=\int_{f^*(t)}^{\lambda_g(t)}(t-g^*(s)) \, \ds\leq \int_{f^*(t)}^{\lambda_g(t)}(t-g^*(\lambda_g(t))) \, \ds\\
& =  \int_{f^*(t)}^{\lambda_g(t)}(t-g^*(\lambda_{g^*}(t))) \, \ds=0,
\end{aligned}
$$
which implies (\ref{goal}). The last equality follows because by the hypotheses on $g^*$, we have  that for every $0<s<t$, $|\{y>0:s<g^*(y)\le t\}|>0$, which is equivalent to the equality $g^*(\lambda_{g^*}(t))=t$. In fact, $g^*(\lambda_{g^*}(t))\le t$ is always true \cite[Proposition~II.1.7]{BS} and $g^*(\lambda_{g^*}(t))\ge t$ means that 
$$
t\le \inf\{s>0:\lambda_{g^*}(s)\le\lambda_{g^*}(t)\};
$$
that is, if $s<t$, then $\lambda_{g^*}(s)>\lambda_{g^*}(t)$, which is the hypothesis. 
\medskip

Similarly, if $f^*(t) > \lambda_g(t)$, (\ref{equiv}) and the monotonicity of $g^*$ imply
$$
\begin{aligned}
\int_0^{f^*(t)} g^*(s) \, \ds-\int_0^{f^*(t)} \lambda_f(s) \, \ds& \leq t(\lambda_g(t)-f^*(t))+\int_{\lambda_g(t)}^{f^*(t)} g^*(s) \, \ds\\
& \leq (f^*(t)-\lambda_g(t))(g^*(\lambda_g(t))-t) \le 0,
\end{aligned}
$$
as desired.
\end{proof}

This last result will also be clear once we prove, in Theorem~\ref{T:GR}, that \eqref{HLP} is equivalent to the existence of a set $E\subset X\times Y$ with $f$ and $g$ as its cross sections. Changing $E$ by its transpose set $\widetilde{E}=\{(y,x):Y\times X, \text{ such that } (x,y)\in E\}$ and applying Proposition~\ref{hlpineq} to $\widetilde{E}$, we finally obtain \eqref{goal}.

\medskip

To finish the section, we are to going to work the details, for  a couple of concrete and elementary  cases, showing that the existence (or, rather, the construction) of the set $E$ is not in general straightforward. Moreover, from these examples we will see that the solution is not, in general, uniquely determined.

\begin{examples}
\

\begin{enumerate}[leftmargin=*]
\item[(i)]
Let $X=Y=[0,1]$, with the Lebesgue measure, and consider $f(x)=g(x)=(1-x)/2$. Then, $f^*(t)=f(t)\chi_{[0,1]}(t)$ and $\lambda_g(t)=(1-2t)\chi_{[0,1/2]}(t)$. Thus,
$$
\int_0^t f(s)\,\ds=\bigg(\frac{t}2-\frac{t^2}4\bigg)\chi_{[0,1}](t)+\frac14\chi_{(1,\infty)}(t)
$$
and
$$
\int_0^t\lambda_g(s)\,\ds=(t-t^2)\chi_{[0,1/2]}(t)+\frac14\chi_{(1/2,\infty)}(t),
$$

\noindent
and \eqref{HLP} holds. For this case, it is easy  to check that any of the sets on Figure~\ref{set1} give a positive solution to Problem~\ref{conj}.

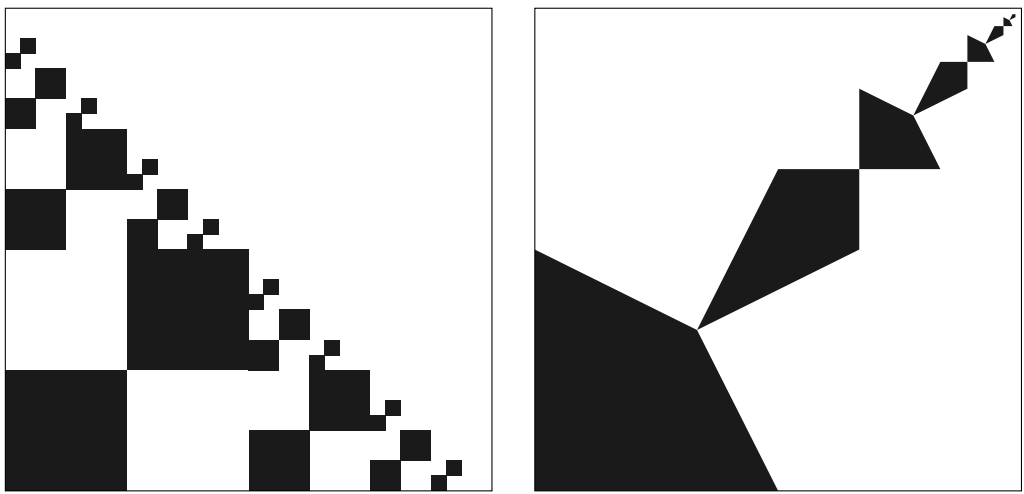
\begin{figure}[htp]
\begin{center}
\begin{tikzpicture}	[scale=0.8]	
	\fill[black!90] (0,0) rectangle (2,2);
	\fill[black!90] (2,2) rectangle (4,4);
	\fill[black!90] (0,4) rectangle (1,5);
	\fill[black!90] (1,5) rectangle (2,6);
	\fill[black!90] (4,0) rectangle (5,1);
	\fill[black!90] (5,1) rectangle (6,2);
	\fill[black!90] (0,6) rectangle (0.5,6.5);
	\fill[black!90] (0.5,6.5) rectangle (1,7);
	\fill[black!90] (6,0) rectangle (6.5,0.5);
	\fill[black!90] (6.5,0.5) rectangle (7,1);
	\fill[black!90] (2,4) rectangle (2.5,4.5);
	\fill[black!90] (2.5,4.5) rectangle (3,5);
	\fill[black!90] (4,2) rectangle (4.5,2.5);
	\fill[black!90] (4.5,2.5) rectangle (5,3);
	\fill[black!90] (0,7) rectangle (0.25,7.25);
	\fill[black!90] (0.25,7.25) rectangle (0.5,7.5);
	\fill[black!90] (7,0) rectangle (7.25,0.25);
	\fill[black!90] (7.25,0.25) rectangle (7.5,0.5);
	\fill[black!90] (1,6) rectangle (1.25,6.25);
	\fill[black!90] (1.25,6.25) rectangle (1.5,6.5);
	\fill[black!90] (6,1) rectangle (6.25,1.25);
	\fill[black!90] (6.25,1.25) rectangle (6.5,1.5);
	\fill[black!90] (2,5) rectangle (2.25,5.25);
	\fill[black!90] (2.25,5.25) rectangle (2.5,5.5);
	\fill[black!90] (5,2) rectangle (5.25,2.25);
	\fill[black!90] (5.25,2.25) rectangle (5.5,2.5);
	\fill[black!90] (3,4) rectangle (3.25,4.25);
	\fill[black!90] (3.25,4.25) rectangle (3.5,4.5);
	\fill[black!90] (4,3) rectangle (4.25,3.25);
	\fill[black!90] (4.25,3.25) rectangle (4.5,3.5);
	\draw[black, thin] (0,0) rectangle (8,8);	
\end{tikzpicture}	\quad\begin{tikzpicture}	[scale=0.8]	
	\fill[black!90] (0,0) -- (0,4) -- (8/3,8/3) -- (4,0);
	\fill[black!90] (8/3,8/3) -- (4,16/3) -- (16/3,16/3) -- (16/3,4);
	\fill[black!90] (16/3,16/3) -- (16/3,20/3) -- (56/9,56/9) -- (20/3,16/3);
	\fill[black!90] (56/9,56/9) -- (60/9,64/9) -- (64/9,64/9) -- (64/9,60/9);
	\fill[black!90] (64/9,64/9) -- (64/9,68/9) -- (200/27,200/27) -- (68/9,64/9);
	\fill[black!90] (200/27,200/27) -- (204/27,208/27) -- (208/27,208/27) -- (208/27,204/27);
	\fill[black!90] (208/27,208/27) -- (208/27,212/27) -- (632/81,632/81) -- (212/27,208/27);
	\fill[black!90] (632/81,632/81) -- (636/81,640/81) -- (640/81,640/81) -- (640/81,636/81);
	\draw[black, thin] (0,0) rectangle (8,8);	
\end{tikzpicture}
\caption{ Two different  approximations of a set $E\subset[0,1]\times[0,1]$,  with cross sections equal to $f(x)=g(x)=(1-x)/2$.}
\label{set1}
\end{center}
\end{figure}

\item[(ii)] For $0<a<1$, let $f(x)=g(x)=a\chi_{[0,1]}(x)$. Then \eqref{HLP} holds and two possible sets for which \eqref{setfg} is satisfied are shown in Figure~\ref{set2}.

\begin{figure}[tbp]
\begin{center}
\begin{tikzpicture}[scale=0.8]							
		\fill[black!90] (0,0) -- (0,8/3) -- (8/3,0);
		\fill[black!90] (8,0) -- (0,8) -- (8/3,8) -- (8,8/3);
		\draw[black, thin] (0,0) rectangle (8,8);	
	\end{tikzpicture}\quad\begin{tikzpicture}	[scale=0.8]	
	\fill[black!90] (0,0) -- (0,8/6) -- (8/6,0);
	\fill[black!90] (8,0) -- (0,8) -- (8/6,8) -- (8,8/6);
	\fill[black!90] (4,0) -- (0,4) -- (0,16/3) -- (16/3,0);
	\fill[black!90] (8,4) -- (4,8) -- (16/3,8) -- (8,16/3);
	\draw[black, thin] (0,0) rectangle (8,8);	
\end{tikzpicture}
\caption{Two different sets $E\subset[0,1]\times[0,1]$,  with cross sections equal to $f(x)=g(x)=a\chi_{[0,1]}(x)$. In this example, $a=1/3$.}
\label{set2}
\end{center}
\end{figure}
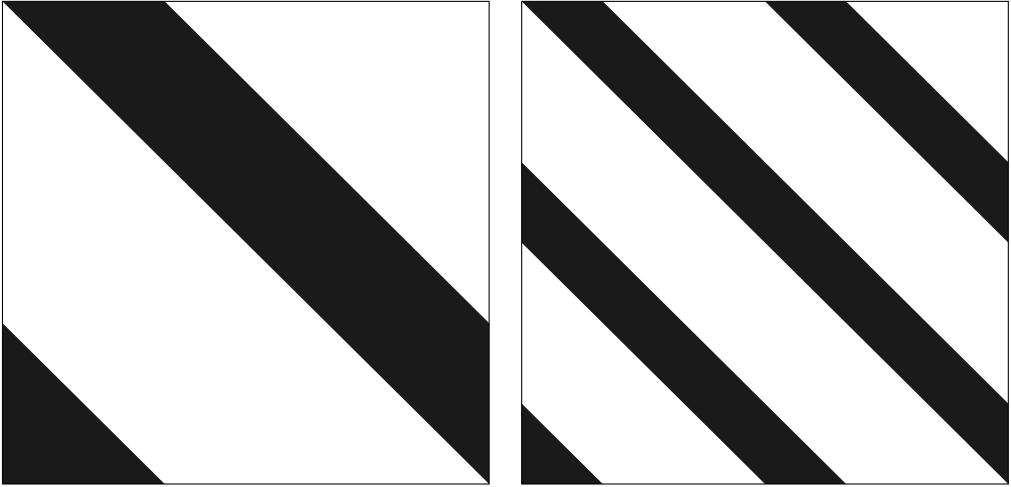

\end{enumerate}

\end{examples}
\bigbreak
\section{Existence of the set $E$ with a priori cross sections}\label{charact}

Our main result in this section is Theorem~\ref{T:GR}, where we show that the necessary condition \eqref{HLP} of Proposition~\ref{hlpineq} is actually sufficient to find a set $E$, in the product space, with given cross sections. The main geometric tool used for such construction is the notion of swappable squares in Definition~\ref{D:sw-sq}, which allows us to horizontally translate the mass of the hypograph of the function $g$ in such a way that in the limit, after suitable iterations for different grids of  dyadic squares, we get precisely a set $E$ with the vertical cross section equal to $f$.  In the discrete  case, this idea lies behind the proof of the Gale-Ryser Theorem~\ref{gary} given in \cite{Krau}.
\medskip

To this end, we will start by proving some interesting properties of this swapping argument, as well as some measure theoretical estimates of the (lower) limit set obtained.

\begin{definition}[Dyadic squares]\label{D:dya}
    For $n\in\N$ and $i,j\in\{1,\ldots,2^n\}$ define
	    \begin{align*}
		    \xnj:= [2^{-n}(j-1),2^{-n}j),\qquad
		    \yni := [2^{-n}(i-1),2^{-n}i)
		\end{align*}
	and
      $$
        Q_{ij}^n := \xnj\times \yni.
      $$
    We call $\qij$ the  \emph{dyadic squares of $n$-th generation} with indices $i,j$.
\end{definition}

\begin{definition}[Shifted set]\label{D:shift}
  Let $Q\subset\R^2$ and $(x_0,y_0)\in\R^2$. Then we define
    $$
      Q + (x_0,y_0) := \{(x,y)\in\R^2: (x-x_0,y-y_0)\in Q\}.
    $$
\end{definition}

\begin{definition}[Cross sections of a~set]\label{D:rearr}
    Let $A\subset [0,1]^2$ be a~measurable set. Define the \emph{vertical cross section} and the \emph{horizontal cross section}, respectively, of the set $A$ as 
     $$   v_A(x)  := \int_0^1 \chi_A(x,z)\dz,\qquad  x\in[0,1],$$
     and
     $$
        h_A(y)  := \int_0^1 \chi_A(z,y)\dz,\qquad  y\in[0,1].$$
\end{definition}

\begin{definition}[Horizontal swapping]\label{D:swap}
  Let $A\subset[0,1]^2$ be a~set. Let $n\in\N$ and $i,j,k\in\{1,\ldots,2^n\}$, $j\ne k$. Then we define the set
    $$
      \sigma^n_{ijk}(A) := (A \setminus (\qij \cup \qik)) \cup ((\qij\cap A) + (2^{-n}(k-j),0)) \cup ((\qik\cap A) - (2^{-n}(k-j),0)).
    $$
  In other words, $\sigma^n_{ijk}(A)$ is the set $A$, whose subsets $\qij\cap A$ and $\qik\cap A$ have ``changed their places''.
\end{definition}

\begin{definition}[Swappable squares]\label{D:sw-sq}
  Let $A\subset[0,1]^2$ be a~measurable set and $f:[0,1]\to [0,1]$ be a~measurable function such that
    $$
      \int_0^t \f(s)\ds \le \int_0^t v_A^*(s) \ds, \quad \text{for all } t\in[0,1].
    $$
  Let $n\in\N$ and $i,j,k \in\{1,\ldots,2^n\}$, $j\ne k$. Then, given $f$ and $n\in\N$ as above,  we say that the dyadic squares $\qij$ and $\qik$ are \emph{swappable} (with respect to $f$, $A$ and $n$)  if, with
    $$
      A' := \sigma^n_{ijk} (A),
    $$
  the following conditions are satisfied:
    \begin{gather}
       v_A \ge f + 2^{-n} \quad \text{a.e.~in } \xnj, \label{E:s1}\\
       v_A \le f - 2^{-n} \quad \text{a.e.~in } \xnk, \label{E:s2}\\
       \int_0^t \f(s)\ds  \le \int_0^t v_{A'}^*(s) \ds \quad \text{for all } t\in[0,1],\label{E:s4}\\
       ((\qik\cap A) - (2^{-n}(k-j),0)) \subsetneq \qij\cap A\quad \text{and}\quad |A\syd A'|>0. \label{E:s5}
    \end{gather}
\end{definition}

\medskip

\begin{lemma}\label{propswap}
Let the function $f$, the indices $n,i,j,k$, the squares $ \qij$ and $ \qik$, and the sets $A$ and $A'$ be as in Definition~\ref{D:sw-sq}. Then,
\begin{enumerate}
\item[\rm(i)] $h_A(y)=h_{A'}(y)$, for every $y\in[0,1]$, and hence $|A|=|A'|$.

\medskip

\item[\rm(ii)]
$$ |A \setminus A'| = \int_{\xnj} (v_{A}(x) - v_{A'}(x)) \dx\ \text{ and }\  |A' \setminus A| = \int_{\xnk} (v_{A'}(x) - v_A(x)) \dx.$$

\item[\rm(iii)]\label{cond3lem}
$$\int_0^1 |f(x)-v_{A'}(x)| \dx =\int_0^1 |f(x)-v_{A}(x)| \dx - |A\syd A'|.$$

\item[\rm(iv)]\label{lemit4}
 If $l\in\{1,\dots,2^n\}$, then:
	\begin{align*}
		f \ge v_{A'} \ge v_{A} \text{ in }X^n_l & \ \text{ if } f \ge v_{A} \text{ in } X^n_l, \\
		f \le v_{A'} \le v_{A} \text{ in } X^n_l& \ \text{ if } f \le v_{A} \text{ in } X^n_l, \\
		v_{A'} = v_{A}   \text{ in } X^n_l & \ \text{ else.}
	\end{align*}

\end{enumerate}
\end{lemma}

\begin{proof}
(i) The result is clear if $y\in[0,1]\setminus Y_i^n$. Now, if $y\in Y_i^n$, then
\begin{align*}
h_{A'}(y)&=\int_0^1\chi_{A'}(x,y)\dx=\int_0^1\chi_{(A \setminus (\qij \cup \qik)) }(x,y)\dx\\
&\qquad+\int_0^1\chi_{((\qij\cap A) + (2^{-n}(k-j),0)) }(x,y)\dx+\int_0^1\chi_{ ((\qik\cap A) - (2^{-n}(k-j),0)) }(x,y)\dx\\
&=\int_0^1\chi_{(A \setminus (\qij \cup \qik)) }(x,y)\dx+\int_0^1\chi_{(\qij\cap A) )}(x,y)\dx+\int_0^1\chi_{ (\qik\cap A)  }(x,y)\dx\\
&=\int_0^1\chi_{A}(x,y)\dx=h_{A}(y).
\end{align*}

\noindent
(ii) We prove the first equality (the second one is completely analogous). Now, using \eqref{E:s5}, we have that:
 \begin{align*}
  |A \setminus A'| & = |(A \setminus A') \cap \qij|
     = \int_{\xnj} \int_{\yni} \chi_{A\setminus A'}(x,y) \dy \dx \\
    &
     = \int_{\xnj} \int_{\yni} (\chi_{A}(x,y) - \chi_{A'}(x,y) )\dy \dx  = \int_{\xnj}( v_{A\cap \qij}(x) - v_{A' \cap \qij} (x)) \dx \\
    &
     = \int_{\xnj} (v_{A}(x) - v_{A'}(x)) \dx.
  \end{align*}

 \noindent
  Let us now prove (iii). Taking into account part (ii),
  \begin{align*}
		&\int_0^1 |f(x)-v_{A'}(x)| \dx \\
		  = &\int_{[0,1]\setminus(\xnj\cup\xnk)} |f(x)-v_{A'}(x)| \dx + \int_{\xnj} |f(x)-v_{A'}(x)| \dx + \int_{\xnk} |f(x)-v_{A'}(x)| \dx \\
		 =& \int_{[0,1]\setminus(\xnj\cup\xnk)} |f(x)-v_{A}(x)| \dx + \int_{\xnj} (v_{A'}(x) - f(x) )\dx + \int_{\xnk}( f(x)-v_{A'}(x) )\dx \\
     =& \int_{[0,1]\setminus(\xnj\cup\xnk)} |f(x)-v_{A}(x)| \dx + \int_{\xnj}( v_{A}(x) - f(x)) \dx + \int_{\xnk}( f(x)-v_{A}(x) )\dx \\
    & \qquad  + \int_{\xnj} (v_{A'}(x) - v_{A}(x) )\dx + \int_{\xnk} (v_{A}(x) - v_{A'}(x)) \dx\\
     =& \int_{[0,1]\setminus(\xnj\cup\xnk)} |f(x)-v_{A}(x)| \dx + \int_{\xnj} |v_{A}(x) - f(x)| \dx + \int_{\xnk} |f(x)-v_{A}(x)| \dx \\
    & \qquad  - | A \setminus A' | - |A'\setminus A|\\
     = &\int_0^1 |f(x)-v_{A}(x)| \dx - |A\syd A'|.
	\end{align*}
	
	\noindent
(iv) Let us observe that, if we construct $A'=\sigma_{ijk}^n(A)$ from $A$, we change the content of the square $\qij$, which is the only change in the $j$-th column. By (\ref{E:s5}), we obtain that $\qij\cap A'\subseteq \qij\cap A$. Hence we get $v_A\geq v_{A'}\geq v_A-2^{-n}$ in $X_j^n$, where the last inequality holds since $2^{-n}$ is the height of the square. Then, condition (\ref{E:s1}) gives
$$v_A\geq v_{A'}\geq v_{A}-2^{-n}\geq f \ \text{in}\ X_j^n.$$
Similarly, by looking at the $k$-th column, since  $\qik\cap A\subseteq \qik\cap A'$, condition (\ref{E:s2}) gives
$$v_A\leq v_{A'}\leq v_{A}+2^{-n}\leq f \ \text{in}\ X_k^n.$$
In any other case, since if neither (\ref{E:s1}) nor (\ref{E:s2}) holds with $l$ in the role of $j$ or $k$, the $l$-th column remains unchanged, and thus $v_A=v_{A'}$ in $X_l^n$.
\end{proof}

\begin{remark}\label{optimset} 
Given $A$, $f,$ and $n$, let $\Omega_n$ be the collection of all sets $B$ obtained from $A$ by a finite number of swappings. It is obvious that $\Omega_n$  is finite, since, because of \eqref{E:s5},  no pair of squares can be swapped twice. Then $A_{{\rm opt},n} $ is defined as such an element of $\Omega_n$  that
$$
\|f -v_{A_{{\rm opt},n}}\|_1 = \min_{B\in \Omega_n} \|f - v_B\|_1. 
$$
\end{remark}

We need to recall the following classical result:

\begin{lemma}\label{T:seq}
	Let  $\{E_n\}_{n\in\N}$ be a~sequence of measurable subsets of $\R^d$ and let
		\begin{equation}\label{E:limit-E}
			E:= \bigcup_{m\in\N} \bigcap_{k=m}^\infty E_k
		\end{equation}
be its lower limit. 	If $|E|<\infty$ and
		\begin{equation}\label{E:meas-sum}
			\sum_{n\in\N} |E_n\setminus E_{n+1}| <\infty,
		\end{equation}
	then $E_n\to E$ in measure; i.e., $|E_n\syd E| \to 0$ as $n\to\infty$. 
	\end{lemma}

Similarly, the proof of the following result is straightforward and follows from the standard properties of the nonincreasing rearrangement of  a function:

\begin{lemma}\label{T:sep} 
  Let $u:[0,1]\to [0,\infty)$ be a~measurable function. Assume that there exist $0<p\leq q <1$ and constants $C_1\ge C_2>0$ such that
    $
      u(x)\ge C_1, 
    $
  for a.e.~$x\in[0,p)$, $C_2 \le u(x) \le C_1$, for a.e.~$x\in[p,q)$ and $u(x)\le C_2$,   for a.e.~$x\in[q,1]$. Then
    $$
      u^*(t) = (u\chi_{[0,p)})^*(t) \text{ for all } t\in[0,p)
    $$
     $$
      u^*(t) = (u\chi_{[p,q)})^*(t-p) \text{ for all } t\in[p,q)
    $$
  and
    $$
      u^*(t) = (u\chi_{[q,1]})^*(t-q) \text{ for all } t\in[q,1].
    $$
  Furthermore, as a~particular consequence, one has
       \begin{align}\label{E:cons}
      \int_0^p u^*(s)\ds = \int_0^p u(z)\dz, &\quad \int_p^q u^*(s)\ds = \int_p^q u(z)\dz,\\
            \text{ and }\quad  \int_q^1 u^*(s)\ds& = \int_q^1 u(z)\dz\nonumber.
\end{align}

\end{lemma}

\medskip

We will now prove our main result. To do this, we   first address the case of cross sections on $[0,1]$ (moreover, under some monotone conditions) and the general setting will then follow using Ryff's Theorem  and some measure preserving transformations (see \cite{ryff} and \cite[Theorem~II.7.5 and Corollary~II.7.6]{BS}).

\begin{theorem}\label{T:GR}
Let $(X,\mu)$ and $(Y,\nu)$ be two finite resonant measure spaces and let $f:X\rightarrow \mathbb R^+$ and $g:Y\rightarrow \mathbb R^+$ be measurable functions satisfying  $\|f\|_{L^1(\mu)}=\|g\|_{L^1(\nu)}$ and
    $$ 
      \int_0^t \f(s)\ds \le \int_0^t \lambda_g(s)\ds,
    $$ 
  for all $t>0$. Then, there exists a~measurable set $E\subset X\times Y$ such that $v_E(x)=f(x)$, $\mu$-a.e. $x\in X$ and $h_E(y)=g(y)$,   $\nu$-a.e. $y\in Y$.
\end{theorem}

\begin{proof} We start by considering the case $X=Y=[0,1]$, endowed with the Lebesgue measure, and assume also that $f$ is a nonincreasing function. Now,   define $E_0$ by
     $$
      E_0:= \left\{ (x,y)\in[0,1]^2: x< g(y)\right\}
    $$
and  $E_n=(E_{n-1})_{{\rm opt},n}$, for each $n\in\N$, as in Remark~\ref{optimset}. 

\medskip

  Observe that since $v_{E_0}=\lambda_g$ in $[0,1]$, hence $v_{E_0}$ is nonincreasing and right-continuous, which in turn implies $v_{E_0} = v_{E_0}^*$ in $[0,1]$.
  Therefore we get
    $$
      \int_0^t \f(s) \ds \le \int_0^t \lambda_g(s)\ds = \int_0^t v_{E_0}^*(s)\ds,
    $$
  for all $t\in[0,1]$. Let $n\in\N$. Then, by the definition of $E_n$,
  \begin{equation}\label{E:fen}
		\int_0^t \f(s) \ds \le \int_0^t v_{E_{n}}^*(s)\ds.
	\end{equation}

\noindent
Using Lemma~\ref{propswap}\,(i), it immediately follows that 
\begin{equation}\label{equalfunct}
|E_n|=|E_{n-1}|, \ h_{E_n}=g \quad\text{a.e. in}\quad [0,1],
\end{equation}
and
\begin{equation}\label{equalnorms}
|E_n|=\|v_{E_n}\|_1=\|h_{E_n}\|_1=\|g\|_1 = \|f\|_1.
\end{equation} 
Furthermore, to fix the notation, suppose now that the set $E_n$ is constructed from $E_{n-1}$ as follows: there exists an $m\in\N$ and sets $A_0, \ldots, A_m$ such that $A_0 = E_{n-1}$, $A_m = E_n$ and $A_{l} = (A_{l-1})'$ for all $l\in\{1,\ldots,m\}$. Then, using  Lemma~\ref{propswap}\,(iii)

\begin{align*}	
		|E_{n-1}\syd E_n|  &\le \sum_{l=1}^m |A_{l-1}\syd A_l| = \sum_{l=1}^m (\|f-v_{A_{l-1}}\|_1 - \|f-v_{A_l}\|_1)\\
		& = \|f-v_{E_{n-1}}\|_1 - \|f-v_{E_n}\|_1.
			\end{align*}
Since the inequality $\|f-v_{A_{l-1}}\|_1 - \|f-v_{A_l}\|_1 > 0$ holds for all $l\in\{1,\ldots,m\}$ (see \eqref{E:s5} and Lemma~\ref{cond3lem}), we have also $\|f-v_{E_{n-1}}\|_1 - \|f-v_{E_n}\|_1 >0$.
Thus, we get
	\begin{align*}
		\sum_{n\in\N} |E_{n-1}\syd E_n| & = \lim_{N\to\infty} \sum_{n=1}^N |E_{n-1}\syd E_n| \\
		& \le \lim_{N\to\infty}  \sum_{n=1}^N \left( \|f-v_{E_{n-1}}\|_1 - \|f-v_{E_n}\|_1 \right)\\
		& = \lim_{N\to\infty} \left( \|f-v_{E_0}\|_1 - \|f-v_{E_N}\|_1 \right) \\
		& \le \|f-v_{E_0}\|_1 \le 1.
	\end{align*}
We finally define $E$ by \eqref{E:limit-E}. Since $E\subset [0,1]^2$ obviously holds, the assumptions of Lemma~\ref{T:seq} are satisfied and it follows that $E_n$ converges to $E$ in measure. Observe that \eqref{equalfunct} and \eqref{equalnorms} give us that  $h_E=g$, a.e. and $\|v_E\|_1=\|f\|_1$.

\medskip

It remains to prove $\|f-v_E\|_1=0$, which will be done in the rest of the proof. At first we need some auxiliary observations. We see that
\begin{equation}\label{norm1ve}
    \|v_{E_n}-v_E\|_1 \le \int_0^1 \int_0^1 |\chi_{E_n}(x,y)-\chi_{E}(x,y)|\dx\dy = |E_n\Delta E|\xrightarrow{n\to\infty} 0.
 \end{equation}
For every $x\in[0,1]$ let $\{j_n(x)\}_{n\in\N}$ be the unique sequence of natural numbers such that $x\in \xnjn$ for each $n\in\N$. Now let $n\in\N$ be fixed. By the construction of $E_n$, for every $i\in\{1,\ldots, 2^n\}$ there exists a~unique permutation $\pi_{ni}: \{1,\ldots, 2^n\} \to \{1,\ldots, 2^n\}$ such that
$E_n \cap \qij = E_0 \cap Q^n_{i\pi_{ni}(j)}$ for all $j\in\{1,\ldots, 2^n\}$. Hence, for $x\in X^n_j$ we have
  $$
    v_{E_n}(x) =  \sum_{i=1}^{2^n} v_{E_n \cap Q^n_{ij}}(x) =  \sum_{i=1}^{2^n} v_{E_0 \cap Q^n_{i\pi_{ni}(j)}}(x).
  $$
By the definition of $E_0$, one has
  $$
    v_{E_0 \cap Q^n_{i\pi_{ni}(j)}}(x) = v_{ \{ (u,z)\in Q^n_{i\pi_{ni}(j)}:u< g(z)\}}(x) = \int_{\yni} \chi_{\{ (u,z)\in Q^n_{i\pi_{ni}(j)}: u< g(z)\}} (x,y) \dy,
  $$
for every pair $i,j\in\{1,\ldots,2^n\}$. Since the above integrand is nonincreasing on the variable $x$, then $x\mapsto v_{E_n \cap Q^n_{ij_n(x)}}(x)=v_{E_0 \cap Q^n_{i\pi_{ni}(j_n(x))}}(x)$ is a~nonincreasing function on $\xnj$, for every $i$ and $j$, and thus $v_{E_n}$ is also nonincreasing on every $\xnj$, $j\in\{1,\ldots,2^n\}$.

It is worth mentioning that, for any $i,j,k \in\{1,\ldots,2^n\}$, the following implication is valid:
  \begin{equation}\label{E:tele}
    \lim_{z\to 2^{-n}j-} v_{E_n \cap \qij}(z) > v_{E_n \cap \qik}(2^{-n}(k-1)) \Longrightarrow  (E_n \cap \qik) - (2^{-n}(k-j),0) \subsetneq E_n \cap \qij.
  \end{equation}
This follows from the representation $E_n \cap \qij = E_0 \cap Q^n_{i\pi_{ni}(j)}$, $E_n \cap \qik = E_0 \cap Q^n_{i\pi_{ni}(k)}$, the definition of $E_0$, and Lemma~\ref{propswap} (iii). 

Next, we have
  \begin{align*}
    & \sum_{j=1}^{2^n} \left( v_{E_n}(2^{-n}(j-1)) - \lim_{z\to 2^{-n}j-} v_{E_n}(z) \right) \\
    & = \sum_{j=1}^{2^n} \sum_{i=1}^{2^n} \left( v_{E_n \cap \qij}(2^{-n}(j-1)) - \lim_{z\to 2^{-n}j-} v_{E_n \cap \qij }(z) \right) \\
    & = \sum_{j=1}^{2^n} \sum_{i=1}^{2^n} \left( v_{E_0 \cap Q^n_{i\pi_{ni}(j)}}(2^{-n}(\pi_{ni}(j)-1)) - \lim_{z\to 2^{-n}\pi_{ni}(j)-} v_{E_0 \cap Q^n_{i\pi_{ni}(j)} }(z) \right)\\
    &= \sum_{j=1}^{2^n} \sum_{i=1}^{2^n} \left( v_{E_0 \cap \qij}(2^{-n}(j-1)) - \lim_{z\to 2^{-n}j-} v_{E_0 \cap \qij }(z) \right),
      \end{align*}
      where the last equality holds by switching the order of the two sums,  rearranging all indices obtained when applying the permutation $\pi_{ni}$ to the set $j\in\{1,\dots,2^n\}$, and switching the indices $i$ and $j$ back again.      Continuing from here we obtain:
 \begin{align*}
     \sum_{j=1}^{2^n} &\left( v_{E_n}(2^{-n}(j-1)) - \lim_{z\to 2^{-n}j-} v_{E_n}(z) \right)\\
    &= \sum_{j=1}^{2^n} \sum_{i=1}^{2^n} \left( v_{E_0 \cap \qij}(2^{-n}(j-1)) - \lim_{z\to 2^{-n}j-} v_{E_0 \cap \qij }(z) \right) \\
    & = \sum_{j=1}^{2^n} \left( v_{E_0}(2^{-n}(j-1)) - \lim_{z\to 2^{-n}j-} v_{E_0}(z) \right) \\
    & = \sum_{j=1}^{2^n} \left( \lambda_g (2^{-n}(j-1)) - \lim_{z\to 2^{-n}j-} \lambda_g (z) \right) \\
    & \le \sum_{j=1}^{2^n} \left( \lambda_g (2^{-n}(j-1)) - \lambda_g (2^{-n}j) \right) \\
    & \le \lambda_g (0).
  \end{align*}
For every $n\in\N$ and $x\in[0,1]$ define
  $$
    \varphi_n(x) := v_{E_n}(2^{-n}(j_n(x)-1)) - \lim_{z\to 2^{-n}j_n(x)-} v_{E_n}(z).
  $$
As observed above, $v_{E_n}$ is nonincreasing on each $\xnj$, therefore $\varphi_n$ is nonnegative on $[0,1]$ and we have
  $$
    \|\varphi_n\|_1 = \sum_{j=1}^{2^n} 2^{-n} \left( v_{E_n}(2^{-n}(j-1)) - \lim_{z\to 2^{-n}j-} v_{E_n}(z) \right) \le 2^{-n} \lambda_g(0) \xrightarrow{n\to\infty} 0.
  $$
Hence, there exists a~subsequence $\{\varphi_{n_m}\}_{m\in\N}$ converging a.e.~in $[0,1]$. In particular, this implies that
  $$
      \liminf_{n\to\infty} \left( v_{E_n}(2^{-n}(j_n(x)-1)) - \lim_{z\to 2^{-n}j_n(x)-} v_{E_n}(z) \right) = 0
  $$
for a.e.~$x\in[0,1]$. By the monotonicity of $v_{E_n}$ on each $\xnj$, we therefore get
  \begin{equation}\label{E:slope}
      \liminf_{n\to\infty} \sup_{z\in X^n_{j(x)}} |v_{E_n}(x) - v_{E_n}(z)| = 0
  \end{equation}
for a.e.~$x\in[0,1]$.

Consequently, using Lemma~\ref{propswap}\,(iv),  for every $x\in[0,1]$ the sequence $\{v_{E_n}(x)\}_{n\in\N}$ is monotone.

Next, for any $n\in\N$ and $t\in[0,1]$ we have
  \begin{align}\label{E:uni}
     \int_0^t |v_{E_n}^*(s) \!-\! v_{E}^*(s)|\ds& \le \|v_{E_n}^*\!-\!v_E^*\|_1 = \|(v_{E_n}^*\!-\!v_E^*)^*\|_1\\
     & \le \|(v_{E_n}\!-\!v_E)^*\|_1 = \|v_{E_n}\!-\!v_E\|_1,\nonumber
  \end{align}
where the second inequality follows from the Lorentz-Shimogaki theorem \cite[Theorem~III.7.4, p.~169]{BS}.
Hence, taking the limit as $n\to\infty$ in \eqref{E:fen}, using \eqref{norm1ve} and \eqref{E:uni} give
  \begin{equation}\label{E:lim-ineq}
    \int_0^t \f(s) \ds \le \int_0^t v_{E}^*(s)\ds,
  \end{equation}
for any fixed $t\in[0,1]$.

Again, using \eqref{norm1ve} and the monotonicity of the sequence $\{v_{E_n}(x)\}_{n\in\N}$, then, the sequence of functions $v_{E_n}$ converges to $v_E$ a.e.~in $[0,1]$. Moreover, the pointwise limit in fact exists for every $x\in[0,1]$, and hence we may thus assume without loss of generality (modifying $E$ on a subset of measure zero, if necessary) that $v_{E_n}$ converges to $v_E$ everywhere in $[0,1]$. Thus, as a consequence of Lemma~\ref{propswap} (iv), we obtain
  \begin{equation}\label{E:mono-lim}
    f(z) \le v_E(z)\le v_{E_n}(z)\le v_{E_0}(z) \text{ \quad or \quad } f(z) \ge v_E(z) \ge v_{E_n}(z) \ge v_{E_0}(z)
  \end{equation}
for all $z\in[0,1]$ and $n\in\N$. In particular, this yields
  \begin{equation}\label{E:sandwich}
    \min\{f(x),v_{E_0}(x)\} \le v_E(x) \le \max\{f(x),v_{E_0}(x)\},\  \text{ for all }x\in[0,1].
  \end{equation}
Define by $S$ the set of all $x\in(0,1)$ such that both $f$ and $v_{E_0}$ are continuous in $x$ and \eqref{E:slope} holds. Since $f$ and $v_{E_0}$ are nonincreasing, and, as shown before, \eqref{E:slope} holds for a.e.~$x\in[0,1]$, it follows that $|[0,1]\setminus S|=0$.

Now we can return to the main point. We are going to show that $f(x)=v_{E}(x)$ for almost all $x\in S$, by contradiction. To do so, assume that the set
  $$
    S_0 := \{ x\in S:v_E(x)>f(x) \}
  $$
has positive measure.

The function $v_{E_0}$ is right-continuous. We may assume that $f$ is right-continuous, otherwise its values may be changed on a~null set. Hence, the level set $V_+:=\{x\in[0,1]: v_{E_0}(x)-f(x)>0\}$ is right-open in the sense that for every $x\in V_+$ we have $(x,x+\delta)\subset V_+$, for some $\delta>0$. By \eqref{E:mono-lim}, inequality $v_{E_0}(x)>f(x)$ holds whenever $v_{E}(x) > f(x)$. Hence $S_0 \subset V_+$. Thus, there exists an~interval $(a,b)\subset V_+\subset[0,1]$ such that
  \begin{alignat}{2}
     v_{E_0}  & \ge v_E  \ge f  &\text{ in }[a,b), \label{E:i1}\\
     \lim_{x\to a-} v_{E_0}(x) ) & \le \lim_{x\to a-}f(x),& \label{E:i2}\\
     v_{E_0}(b)  & \le f(b),& \label{E:i3}\\
 |S_0 \cap (a,b)|&  >0.& \label{E:i4}
  \end{alignat}
Notice that $(a,b)\subset[0,1]$ and hence $b<1$. If $a>0$, let $x\in[0,a)$ be arbitrary. Using \eqref{E:sandwich}, \eqref{E:i2} and monotonicity of $v_{E_0}$, we have
  $$
    v_{E}(x) \ge \min\{f(x),v_{E_0}(x)\} \ge \lim_{y\to a-} \min\{f(y),v_{E_0}(y)\} = \lim_{y\to a-} v_{E_0}(y) \ge v_{E_0}(a).
  $$
Similarly, for any $y\in[a,1]$ we have
  $$
    v_{E}(y) \le \max\{f(y),v_{E_0}(y)\}\le  \max\{f(a),v_{E_0}(a)\} = v_{E_0}(a).
  $$
Analogously, by \eqref{E:sandwich}, \eqref{E:i1} and \eqref{E:i3} we get, for any $y\in[0,b)$,
  $$
    v_{E}(y) \ge \min\{f(y),v_{E_0}(y)\} \ge \lim_{z\to b-} \min\{f(z),v_{E_0}(z)\} = \lim_{z\to b-} f(z) \ge f(b) \ge v_{E_0}(b).
  $$
For every $z\in[b,1]$ one has
  $$
    v_{E}(z) \le \max\{f(z),v_{E_0}(z)\}\le   \max\{f(b),v_{E_0}(b)\}  =f(b).
  $$
Hence, we have
  $$
    v_E(x) \ge v_{E_0}(a) \ge v_E(y) \ge  f(b) \ge v_E(z)
  $$
for all $x\in [0,a)$, $y\in[a,b)$ and $z\in [b,1]$. For the sake of correctness, we note that if $a=0$, the first term and inequality is simply omitted. Lemma \ref{T:sep} now gives
  $$
    v_E^*(t) =  (v_E\chi_{[b,1]})^*(t-b)\text{ \ for all \ } t\in[b,1],
  $$
and
  $$
    \int_a^b v_{E}^*(s)\ds =  \int_a^b v_{E}(y)\dy.
  $$
By \eqref{E:i4}, one has $v_E>f$ on a~subset of $(a,b)$ with positive measure. Thus, there exists a~$\varrho>0$ such that
  $$
    \int_a^b v_{E}(y)\dy > \int_a^b f(y)\dy + 4\varrho = \int_a^b \f(s)\ds + 4\varrho.
  $$
From this and \eqref{E:lim-ineq} used with $t=a$, we obtain
  \begin{align}\label{E:cruc}
    \int_0^b \f(s) \ds + 4\varrho &= \int_0^a \f(s) \ds + \int_a^b \f(s) \ds + 4\varrho\\
    & < \int_0^a v_{E}^*(s)\ds + \int_a^b v_{E}^*(s)\ds = \int_0^b v_{E}^*(s)\ds.\nonumber
  \end{align}
Define
  $$
    h:= \min\left\{ t\in(b,1): \int_0^t v_E^*(s)\ds \le \int_0^t \f(s)\ds + 3\varrho \right\}.
  $$
The minimum is indeed attained since the function $t\mapsto \int_0^t v_E^*(s)\ds - \int_0^t \f(s)\ds$ is continuous and  $\int_0^1 v_E^*(s)\ds  = \int_0^1 \f(s)\ds$, which is \eqref{equalnorms}.
Let us show that $f>v_E$ holds on a~subset of $(b,h)$ with positive measure. To do so, assume, for contradiction, that $f\le v_E$ a.e.~in $[b,h)$. Let $t\in[b,h)$. The assumption implies $(f\chi_{[b,h)})^*(t-b) \le (v_E\chi_{[b,h)})^*(t-b)$.
Since $f$ is nonincreasing, by Lemma \ref{T:sep} we have
  $$
     (f\chi_{[b,h)})^*(t-b) = (f\chi_{[b,1]})^*(t-b) = \f(t).
  $$
Consequently,
  $$
    v_E^*(t) = (v_E\chi_{[b,1]})^*(t-b) \ge (v_E\chi_{[b,h)})^*(t-b) \ge (f\chi_{[b,h)})^*(t-b) = \f(t).
   $$
Therefore we have $v_E^*(t) \ge \f(t)$ for all $t\in(b,h)$. However, this inequality together with \eqref{E:cruc} implies
  \begin{align*}
    \int_0^h v_E^*(s)\ds &= \int_0^b v_E^*(s)\ds + \int_b^h v_E^*(s)\ds\\
    & > \int_0^b \f(s)\ds + \int_b^h \f(s)\ds + 4\varrho = \int_0^h \f(s)\ds + 4\varrho.
 \end{align*}
This contradicts the definition of $h$.

We have shown that $|\{ y\in (b,h): f(y)>v_E(y) \}|>0$. Since $|[0,1]\setminus S|=0$, we have also $|\{ y\in (b,h)\cap S: f(y)>v_E(y) \}|>0$.
Thanks to this and \eqref{E:i4}, there exist points $x\in(a,b)\cap S_0$ and $y\in (b,h)\cap S$ such that
  \begin{equation}\label{eqe0f}
    v_E(x) > f(x) \ge f(y) > v_E(y)\ge v_{E_0}(y).
  \end{equation}
Define 
$$
\eps:= \frac14 \min \{ v_E(x)- f(x),\, f(y) - v_E(y) \}.
$$
Since $f$ and $v_{E_0}$ are both continuous in $x$ as well as in $y$ (recall the definition of $S$),  there exists a $\delta>0$ satisfying 
\begin{equation}\label{deltaincl}
\min\{x-a,2(b-x),y-b,1-y\}>\delta>0
 \end{equation}
 such that
  \begin{align}\label{first34}
    |v_{E_0}(z)-v_{E_0}(x)| < \eps \text{  and  } |f(z)-f(x)| < \eps  & \text{ for all }z\in(x-\delta, x+\delta),\nonumber\\[-.3cm] \\[-.3cm]
    |v_{E_0}(z)-v_{E_0}(y)| < \eps \text{  and  } |f(z)-f(y)| < \eps  & \text{ for all }z\in(y-\delta, y+\delta).\nonumber
  \end{align}
In particular, we get  $v_{E_0}>f$ and $f> v_{E_0}$ on $(x-\delta, x+\delta)$ and $(y-\delta, y+\delta)$, respectively. 

By \eqref{E:uni} we have, for $t\in[0,1]$,
  $$
    \int_0^t |v_{E_n}^*(s) - v_E^*(s)| \ds \xrightarrow{n\to\infty} 0
  $$
and this convergence is uniform in $t$. Moreover, since $x,y\in S$, both $x$ and $y$ satisfy \eqref{E:slope} (in the latter case with $y$ instead of $x$) and $f$ as well as $v_{E_0}$ are continuous in $x$ and $y$.
Based on these properties, there exists a~sufficiently large $n\in \N$ such that all the following conditions are satisfied:
 \begin{align}\label{second34}
    \sup_{z\in X^n_{j_n(x)}} |v_{E_n}(x) - v_{E_n}(z)| < \eps,& \quad  \sup_{z\in X^n_{j_n(y)}} |v_{E_n}(y) - v_{E_n}(z)| < \eps,\nonumber\\[-.4cm] \\[-.35cm]
    2^{-n-1}< \eps, \quad 2^{1-n}&<\delta, \quad 2^{1-2n}<\varrho,\nonumber
  \end{align}
and
  \begin{equation}\label{E:futral}
    \int_0^t |v_{E_n}^*(s)- v_E^*(s)| \ds < \varrho \quad \text{for all } t\in[0,1].
  \end{equation}
In the following, we will write $j:= j_n(x)$ and $k:= j_n(y)$ since $n$, $x$ and $y$ remain fixed.
By \eqref{E:mono-lim}, we have
  $$
    v_{E_0}(x)\ge v_{E_n}(x) \ge v_E(x) \ge f(x) + 4\eps.  $$
Let $z\in \xnj$ be arbitrary. We have
  \begin{align*}
    v_{E_n}(z) & \ge v_{E_n}(x) - |v_{E_n}(x) - v_{E_n}(z)| > v_{E_n}(x) - \eps = f(x) + v_{E_n}(x)  - f(x) - \eps \\
    & \ge f(x) + 3 \eps \ge f(2^{-n}(j-1))- |f(2^{-n}(j-1))-f(x)| + 3\eps\\
    & \ge f(2^{-n}(j-1))+2\eps.
  \end{align*}
Taking the limit (or infimum), we get
  \begin{equation}\label{E:dif-j}
    \inf_{z\in\xnj} v_{E_n}(z) =  \lim_{z\to 2^{-n}j-} v_{E_n}(z) \ge f(2^{-n}(j-1))+2\eps.
  \end{equation}
Proceeding analogously regarding the point $y$, we obtain
  \begin{equation}\label{E:dif-k}
    \inf_{z\in\xnk} f(z) = \lim_{z\to 2^{-n}k-} f(z) \ge v_{E_n}(2^{-n}(k-1))+2\eps.
  \end{equation}
Taking into account the inequalities
  $$
    f(2^{-n}(j-1)) \ge f(b) \ge  f(y)\ge \lim_{z\to 2^{-n}k-} f(z),
  $$
we have
  \begin{equation}\label{E:roz}
     \lim_{z\to 2^{-n}j-} v_{E_n}(z) \ge v_{E_n}(2^{-n}(k-1)) +  4\eps.
  \end{equation}

Recall that, for any $i\in\{1,\ldots,2^n\}$, the functions $v_{E_n\cap\qij}$ and $v_{E_n\cap \qik}$ are nonincreasing on $\xnj$ and $\xnk$, respectively.
Inequality \eqref{E:roz} thus yields
  $$
     \sum_{i=1}^{2^n} \lim_{z\to 2^{-n}j-} v_{E_n\cap\qij}(z) > \sum_{i=1}^{2^n} v_{E_n\cap\qik}(2^{-n}(k-1)).
  $$
Hence, there exists an $i\in\{1,\ldots,2^n\}$ such that
  \begin{equation}\label{E:pretele}
     \lim_{z\to 2^{-n}j-} v_{E_n\cap\qij}(z) > v_{E_n\cap\qik}(2^{-n}(k-1)).
  \end{equation}
If we are able to show that the squares $\qij$ and $\qik$ are swappable with respect to $E_n$ and $n$, we will obtain the ultimate contradiction with the construction of $E_n$.

Thus, we set $A:=E_n$ and $A':=\sigma_{ijk}^n(A)$ and want to show that \eqref{E:s1}--\eqref{E:s5} holds. By \eqref{E:pretele} and \eqref{E:tele} we immediately obtain condition \eqref{E:s5}. Since $2^{-n-1}<\eps$ was assumed and $f$ is nonincreasing, estimate \eqref{E:dif-j} implies
  $$
    \inf_{z\in\xnj} v_{E_n}(z) > \max_{z\in\xnj} f(z)+2^{-n},
  $$
Therefore \eqref{E:s1} is satisfied. Analogously, \eqref{E:dif-k} yields \eqref{E:s2}.
\medskip

It remains to prove \eqref{E:s4}. In this part, we will frequently use the inequality
  \begin{equation}\label{E:aa}
    \int_0^t v_A^*(s) \ds \ge \int_0^t \f(s)\ds \text{ \ for all } t\in[0,1],
  \end{equation}
 which obviously holds since $A = E_n$.

Since, by \eqref{eqe0f}, $v_{E_0}(y) < f(y)$, we can define $(c,d)$ as the maximal open subinterval of $(b,1)$ such that $y\in(c,d)$ and  $v_{E_0} < f$ in $(c,d)$. Moreover, using \eqref{first34} and \eqref{second34}, we have that 
\begin{equation}\label{subsetsint}
X^n_j\subset (a,b)\quad\text{ and }\quad X^n_k\subset (c,d).
\end{equation}
Keeping in mind that $A = E_n$, by applying \eqref{E:mono-lim} and \eqref{E:sandwich} to \eqref{E:i1}--\eqref{E:i3} we get the estimates
  \begin{equation}\label{E:ab}
    v_{A} \ge v_{E_0}(a) \text{ in } [0,a), \quad v_{E_0}(a)\ge v_{A} \ge f \ge  f(b) \text{ in } [a,b), \quad   f(b)\ge v_{A} \text{ in }[b,1].
  \end{equation}
Now we may proceed analogously with the interval $(c,d)$. Its maximality and \eqref{E:mono-lim} guarantee that
  \begin{equation}\label{E:cd}
   v_{A} \ge f(c) \text{ in } [0,c), \quad f(c) \ge f \ge   v_{A} \ge v_{E_0}(d)\text{ in } [c,d), \quad  v_{E_0}(d)\ge v_{A} \text{ in }[d,1].
  \end{equation}
 Notice that \eqref{E:ab} and \eqref{E:cd} hold also with $A'$ in place of $A$. To see this, recall that $A'=\sigma^n_{ijk}(A)$ and then use Lemma~\ref{propswap} (iv) and \eqref{subsetsint} which implies
  \begin{equation}\label{E:est-sgm}
    |v_A - v_{A'}| \le 2^{-n}(\chi_{\xnj} + \chi_{\xnk}).
  \end{equation}
Based on \eqref{E:ab}, \eqref{E:cd}  and their analogues with $A'$, by  Lemma \ref{T:sep}  we obtain:
  \begin{equation}\label{E:split}
    v_A^*(t) = (v_A\chi_{[p,q)})^*(t-p) \text{ \ and \ } v_{A'}^*(t) = (v_{A'}\chi_{[p,q)})^*(t-p) \text{ \ for \ } t\in[p,q),
  \end{equation}
whenever $[p,q)$ is one of the intervals $[0,a),[a,b),[b,c),[c,d),[d,1)$. If $a=0$ or $b=c$, the intervals $[0,a)$ or $[b,c)$ are, of course, not considered.

In particular, using \eqref{E:split} with $[p,q)=[a,b)$, $A'=\sigma^n_{ijk}(A)$ and  $\xnj\subset (a,b)$, which is \eqref{subsetsint}, we have
  \begin{align}\label{E:0b}
    \int_0^b v_{A'}^*(s)\ds& = \int_0^b v_{A'}(z)\dz = \int_0^b v_{A}(z) \dz - |A\setminus A'| \\
    &=  \int_0^b v_{A}^*(s)\ds - |A\setminus A'|.\nonumber
  \end{align}
Again, using \eqref{subsetsint} we have that $\xnj\subset (a,b)$ and  $\xnk \subset (c,d)$, and obtain
  \begin{equation}\label{E:eq}
     v_A = v_{A'} \text{ in } [0,a) \cup [b,c) \cup [d,1].
  \end{equation}
If $a>0$, let $t\in[0,a]$. By \eqref{E:eq}, \eqref{E:split} with $[p,q)=[0,a)$, and \eqref{E:aa}, we get
  \begin{equation}\label{E:fin}
    \int_0^t v_{A'}^*(s) \ds = \int_0^t v_{A}^*(s) \ds \ge \int_0^t \f(s) \ds.
  \end{equation}
Let $t\in[a,b]$. Since $v_{A'}\ge f$ in $[a,b)$ (recall that this is \eqref{E:ab} with $A$ replaced by $A'$), \eqref{E:split} with $[p,q)=[a,b)$, and \eqref{E:aa} provide
 \begin{align*}
    \int_0^t (v_{A'}^*(s)-\f(s)) \ds& = \int_0^a (v_{A'}^*(s)-\f(s)) \ds \\
    &\qquad+ \int_a^t ((v_{A'}\chi_{(a,b)})^*(s-a)-(f\chi_{(a,b)})^*(s-a)) \ds  \ge 0,
\end{align*}
thus \eqref{E:fin} holds for $t\in[a,b]$.

If $b<c$, let $t\in[b,c]$. Since $ c<h$, the definition of $h$ and \eqref{E:futral} yield
  $$
     \int_0^t  v_{A}^*(s)\ds > \int_0^t \f(s)\ds + 2\varrho.
  $$
Property \eqref{E:eq} yields $(v_{A'}\chi_{[b,c)})^*= (v_{A}\chi_{[b,c)})^*$. Considering \eqref{E:split} with $[p,q)=[b,c)$, \eqref{E:0b}, the assumption $\varrho > 2^{-2n}$, and $\eqref{E:aa}$, one has
  \begin{align*}
    \int_0^t  v_{A'}^*(s)\ds & = \int_0^b  v_{A'}^*(s)\ds + \int_b^t  v_{A'}^*(s)\ds
     = \int_0^b v_{A'}(z)\dz + \int_b^t  (v_{A'}\chi_{[b,c)})^*(s-b) \ds\\
    & = \int_0^b v_{A}(z)\dz - |A\setminus A'| + \int_b^t  (v_{A}\chi_{[b,c)})^*(s-b) \ds \\
    & = \int_0^t v_A^*(s) \ds - |A\setminus A'| > \int_0^t \f(s)\ds + 2\varrho - |A\setminus A'| \\
    & \ge \int_0^t \f(s)\ds + 2\varrho - 2^{-2n}
     > \int_0^t \f(s)\ds + \varrho.
  \end{align*}
Hence, \eqref{E:fin} is satisfied for $t\in[b,c]$. In particular, the previous calculation gives
  $$ 
    \int_0^c v_{A'}^*(s)\ds = \int_0^c v_A^*(s) \ds - |A\setminus A'|.
  $$ 
Moreover, since $A'=\sigma^n_{ijk}(A)$ and $\xnk \subset (c,d)$, by \eqref{E:split} with $[p,q)=[c,d)$ we have
  $$
    \int_c^d v_{A'}^*(s)\ds = \int_c^d v_{A'}(z)\dz = \int_c^d v_A(z)\dz +  |A'\setminus A| = \int_c^d v_{A}^*(s)\ds + |A'\setminus A|.
  $$
Thanks to these two relations, we have
  \begin{equation}\label{E:eq-d}
    \int_0^d v_{A'}^*(s)\ds = \int_0^c v_{A'}^*(s)\ds + \int_c^d v_{A'}^*(s)\ds = \int_0^d v_A^*(s) \ds.
  \end{equation}
Let $t\in[c,d]$. By  \eqref{E:est-sgm} it follows that $f\ge v_{A'}$ holds in~$[c,d)$,
 and  therefore $(f\chi_{[c,d)})^* \ge (v_{A'}\chi_{[c,d)})^*$. Applying this to \eqref{E:split} with $[p,q)=[c,d)$, one has $\f \ge v^*_{A'}$ in $[c,d)$. Using this fact, \eqref{E:eq-d} and \eqref{E:aa}, we get
  \begin{align*}
    \int_0^t v_{A'}^*(s)\ds & = \int_0^d v_{A'}^*(s)\ds + \int_t^d (\f(s) - v_{A'}^*(s)) \ds - \int_t^d \f(s)\ds \\
    & \ge \int_0^d v_A^*(s)\ds - \int_t^d \f(s)\ds \ge \int_0^d \f(s)\ds - \int_t^d \f(s)\ds\\
    & = \int_0^t \f(s) \ds.
  \end{align*}
In other words, \eqref{E:fin} holds for $t\in[c,d]$.

If $d<1$, let $t\in[d,1]$. By \eqref{E:eq} we have $(v_{A'}\chi_{(d,1)})^* = (v_{A}\chi_{(d,1)})^*$ which together with \eqref{E:split} for $[p,q)=[d,1)$, \eqref{E:eq-d} and \eqref{E:aa} provides
  \begin{align*}
    \int_0^t v_{A'}^*(s) \ds & = \int_0^d v_{A'}^*(s) \ds + \int_d^t v_{A'}^*(s) \ds = \int_0^d v_{A}^*(s) \ds + \int_d^t (v_{A'}\chi_{(d,1)})^*(s) \ds \\
    & =  \int_0^d v_{A}^*(s) \ds + \int_d^t (v_{A}\chi_{(d,1)})^*(s) \ds = \int_0^t v_{A}^*(s) \ds \ge \int_0^t \f(s)\ds.
  \end{align*}
This shows that \eqref{E:fin} is satisfied for $t\in[d,1]$. If $d=1$, this step is omitted.

Altogether we have shown that \eqref{E:fin} is valid for all $t\in[0,1]$, hence condition \eqref{E:s4} is satisfied. Thus, all conditions \eqref{E:s1}--\eqref{E:s5} are met, which means that $\qij$ and $\qik$ are swappable with respect to $A$ and $n$.

Let us summarize what we have proven at this point. Having started with the assumption $|\{ x\in S: v_E(x)>f(x) \}|>0$, we have found a~set $E_n$ such that there exist swappable squares $\qij$, $\qik$ with respect to $E_n$ and $n$. This contradicts the construction of $E_n$, as explained in Remark~\ref{optimset}.

Recalling that $|[0,1]\setminus S|=0$, we have therefore $v_E\le f$ a.e.~in $[0,1]$. Since $\|v_E\|_1=\|f\|_1$ and both functions are nonnegative, it follows that $v_E = f$ a.e.~in $[0,1]$. 

\medskip
The proof is now complete, under the conditions $X=Y=[0,1]$ and $f$ nonincreasing. If $X$ and $Y$ are general finite resonant measure spaces (we can assume, without loss of generality, that they are probability spaces), we use Ryff's Theorem  \cite[Theorem~II.7.5 and Corollary~II.7.6]{BS} to find a couple of measure preserving transformations
$$
\sigma_\mu: X\rightarrow [0,1]\quad\text{ and }\quad\sigma_\nu: Y\rightarrow [0,1],
$$
such that $f=f^*\circ\sigma_\mu$, $\mu$-a.e. and $g=g^*\circ\sigma_\nu$, $\nu$-a.e. We now apply the previous case to the functions $f^*$ and $g^*$ on $[0,1]$ (observe that $f^*$ is nonincreasing and the hypothesis on the primitives of $f^*$ and $\lambda_{g^*}$ trivially holds), to obtain a set $\widehat E\subset [0,1]^2$  such that $h_{\widehat E}=g^*$ and $v_{\widehat E}=f^*$, a.e. We now define the measurable set $E\subset X\times Y$ as follows:
$$
E=\{(x,y)\in X\times Y: (\sigma_\mu(x),\sigma_\nu(y))\in\widehat E\}.
$$
Then,  using \cite[Proposition~II.7.2]{BS}
\begin{align*}
h_E(y)&=\int_X\chi_E(x,y)\,d\mu(x)=\int_X\chi_{\widehat E}(\sigma_\mu(x),\sigma_\nu(y))\,d\mu(x)\\
&=\int_0^1\chi_{\widehat E}(s,\sigma_\nu(y))\ds=h_{\widehat E}(\sigma_\nu(y))=g^*((\sigma_\nu(y)))=g(y),\quad \nu\text{-a.e.}
\end{align*}
Analogously, one can prove that $v_E(x)=f(x)$, $\mu$-a.e.
\end{proof}

\begin{remark}
As   direct consequences of Theorem~\ref{T:GR}, with the Lebesgue measure on $\mathbb R$ and the counting measures on finite sets, we recover Lorentz's result, Theorem~\ref{planeset}, as well as Gale--Ryser's Theorem~\ref{gary}        for matrices in the discrete setting.
\end{remark}
\medskip

\noindent
\textbf{Acknowledgment.} We would like to thank the anonymous referee for some important comments and clarifications, which have greatly improved the final exposition of this work.

\end{document}